\documentclass[a4paper,11pt]{amsart}
\usepackage{amssymb}
\usepackage{mathrsfs}
\usepackage{amsmath,amssymb,amsthm,latexsym,amscd,mathrsfs}
\usepackage{indentfirst}
 \newcommand{\ROM}[1]{\mathrm{\uppercase\expandafter{\romannumeral#1}}}
  \theoremstyle{definition}
 \newtheorem{thm}{Theorem}[section]
 \newtheorem{lem}{Lemma}[section]
 
 \newtheorem{cor}{Corollary}[section]
  
 \newtheorem{rem}{Remark}[section]
 \newtheorem{prop}{Proposition}[section]

\newtheorem{ack}{Acknowledgements}   

\title[Geometry of isoparametric hypersurfaces in Riemannian manifolds]{\textbf{Geometry of isoparametric hypersurfaces in Riemannian manifolds}}
\author[J. Q. Ge]{Jianquan Ge}\address{School of Mathematical Sciences, Laboratory of Mathematics and Complex Systems, Beijing Normal
University, Beijing 100875}\email{jqge@bnu.edu.cn}
\thanks {The project is partially supported by the NSFC ( No.11071018 and No.11001016
), the SRFDP (No.
20100003120003), and the Program for Changjiang Scholars and Innovative
Research Team in University.}
\author[Z. Z. Tang]{Zizhou Tang}\address{School of Mathematical Sciences, Laboratory of Mathematics and Complex Systems, Beijing Normal
University, Beijing 100875}\email{zztang@mx.cei.gov.cn}
\date{}
\begin{document}
\maketitle
\begin{abstract}
In our previous work, we studied isoparametric functions on Riemannian
manifolds, especially on exotic spheres. One result there says that,
in the family of isoparametric hypersurfaces of a closed Riemannian
manifold, there exists at least one minimal isoparametric
hypersurface. In this paper, we show such a minimal isoparametric
hypersurface is also unique in the family if the ambient manifold
has positive Ricci curvature. Moreover, we give a proof of
Theorem D claimed by Q.M.Wang (without proof) which asserts that the
focal submanifolds of an isoparametric function on a complete
Riemannian manifold are minimal. Further, we study isoparametric
hypersurfaces with constant principal curvatures in general
Riemannian manifolds. It turns out that in this case the focal
submanifolds have the same properties as those in the standard sphere,
\emph{i.e.}, the shape operator with respect to any normal direction
has common constant principal curvatures. Some necessary conditions
involving Ricci curvature and scalar curvature are also derived.
\end{abstract}

\section{Introduction}
A non-constant smooth function $f: N\rightarrow \mathbb{R}$ defined
on a smooth connected Riemannian manifold $N$ is called \emph{transnormal} if there
is a smooth function $b: J\rightarrow\mathbb{R}$ such that
\begin{equation}\label{iso1}
|\nabla f|^2=b(f),
\end{equation}
where $J=f(M)\subset \mathbb{R}$ and $\nabla f$ is the gradient of $f$. If moreover there is a continuous function $a: J\rightarrow\mathbb{R}$ such that
\begin{equation}\label{iso2}
\triangle f=a(f),
\end{equation}
where $\triangle f$ is the Laplacian of $f$, then $f$ is called
\emph{isoparametric} (cf. \cite{Wa87}, \cite{GT09}). Equation
(\ref{iso1}) means that the regular hypersurfaces $M_t:=f^{-1}(t)$
(where $t$ is a regular value of $f$) are parallel and (\ref{iso2})
says that these parallel hypersurfaces have constant mean
curvatures. These regular level hypersurfaces $M_t:=f^{-1}(t)$ of an
isoparametric function $f$ are called \emph{isoparametric
hypersurfaces}. A transnormal function $f$ on a complete Riemannian manifold has no critical value in $Int(J)$ (\cite{Wa87}). The preimage of the maximum (resp. minimum), if it exists, of an
isoparametric (or transnormal) function $f$ is called the
\emph{focal variety} of $f$, denoted by $M_{+}$ (resp.
$M_{-}$).\footnote{In our papers, for simplicity we always assume in
the definition that $b$ is smooth, \emph{i.e.}, $C^{\infty}$. In
fact, all the results go through when $b$ is merely $C^2$ as
remarked in \cite{Wa87}. In private discussions, G. Thorbergsson
pointed out that it should also work for $b$ being a continuous
function in the case of both equations (\ref{iso1}-\ref{iso2}) are
satisfied, since then there should exist some other function
$\tilde{f}$ with the same level sets as $f$ such that it satisfies
equations (\ref{iso1}-\ref{iso2}) for some smooth ($C^2$) function
$b$.}

A fundamental structural result given by \cite{Wa87} says that the
focal varieties $M_{\pm}$ of a transnormal function on a complete Riemannian manifold are smooth
submanifolds and each regular level hypersurface $M_t$ is a tubular hypersurface
over either of $M_{\pm}$. In our previous work \cite{GT09}, we
called an isoparametric (or transnormal) function \emph{proper} if
each component of $M_{\pm}$ has codimension not less than 2. For a
properly isoparametric function, $M_{\pm}$ are just the focal sets
of each regular level hypersurface $M_t$ and all level sets are
connected. Moreover, there we proved that at least one properly
isoparametric hypersurface is minimal if the ambient space $N$ is
closed. In this paper, by using the Riccati equation, we observe
that such a minimal isoparametric hypersurface is also unique in its
family if $N$ has positive Ricci curvature. Next, we express the
shape operator $S(t)$ of $M_t$ as a power series with respect to the
distance between $M_t$ and $M_{\pm}$. As an immediate consequence,
we can give a complete proof to Theorem D of \cite{Wa87}
(without proof there):
\begin{thm}\label{minimal focal}\footnote{Note added in proof: Very recently we happened to see an unpublished paper by Ni \cite{Ni97} from internet which gave another different proof of Theorem \ref{minimal focal}.}
The focal varieties $M_{\pm}$ of an isoparametric function $f$ on a
complete Riemannian manifold $N$ are minimal submanifolds.
\end{thm}
In fact, by combining Wang's structural result \cite{Wa87}, one can see
that under an additional assumption that $N$ and $M_{\pm}$ are all
compact, Mazzeo and Pacard \cite{MP05} have essentially proved the
above theorem as a special case in their Theorem 6.1 by an argument
from geometric measure theory which is not fit for the noncompact case.
In private communications, we learned that Miyaoka (cf. \cite{Mi12}) is also
concerned with a complete proof of Theorem \ref{minimal focal}.
Further study of the power series expression establishes
\begin{thm}\label{austere focal}
Suppose that each isoparametric hypersurface $M_t$ has constant
principal curvatures\footnote{When $M_t$ is disconnected, we assume that all connected components have common constant principal curvatures, or equivalently, the principal curvatures of $M_t$ depend only on $f|_{M_t}=t$.}\label{constantcommon} with respect to the unit normal vector field in the direction of $\nabla f$. Then each of the focal varieties $M_{\pm}$ has
common constant principal curvatures in all normal directions,
\emph{i.e.}, the eigenvalues of the shape operator are constant and independent of the choices of the point and unit normal
vector of $M_{\pm}$.
\end{thm}
As a corollary, the principal curvatures of such focal varieties
occur as pairs of opposite signs and thus $M_{\pm}$ are
\emph{austere} (minimal) submanifolds in the sense of \cite{HL82}.
As byproducts, we obtain some necessary curvature conditions for a
manifold to admit (certain) isoparametric functions in Corollary
\ref{minimal curv rest} and Corollary \ref{austere curv rest} in
Section \ref{proof}.

Submanifolds with constant principal curvatures are regarded as the
analogue in submanifold geometry to locally symmetric spaces in
Riemannian geometry in \cite{BCO03} where the normal holonomy theory
takes important role in the studies of such submanifolds in space
forms. It is worth mentioning that the normal holonomy theory for
submanifolds in general Riemannian manifolds has not been fully
developed though it seems fruitful in many areas (cf. \cite{Br99},
\cite{CDO08}, \emph{etc.}). Note that Theorem \ref{minimal focal}
and Theorem \ref{austere focal} generalize corresponding results in
the classical theory of isoparametric hypersurfaces in real space forms
(cf. \cite{No73}, \cite{Mu80}), recalling that, a hypersurface $M^n$ in
a real space form $N^{n+1}(c)$ with constant sectional curvature $c$
is said to be \emph{isoparametric} if it has constant principal
curvatures. Cartan (\cite{Ca38},\cite{Ca39}) and M{\"u}nzner
(\cite{Mu80}) showed that such an isoparametric hypersurface belongs to
a family of parallel hypersurfaces of constant mean curvature which
consists of regular level sets of a homogeneous polynomial (called
\emph{isoparametric polynomial}) satisfying the so-called
Cartan-M{\"u}nzner equation. Provided with the ``good symmetry" of
space forms, M{\"u}nzner \cite{Mu80} could obtain more specific
properties about principal curvatures of the focal varieties and
isoparametric hypersurfaces than Theorem \ref{austere focal} does,
on which they were heavily depended in the proof of his splendid
result about the number $g$ of distinct principal curvatures of
isoparametric hypersurfaces in spheres (For the theory of
isoparametric hypersurfaces in the model space forms, see \cite{Th00} for an excellent survey
and see \cite{CCJ07}, \cite{Imm08}, \cite{GX10} for recent
progresses and applications).

\section{The shape operator of Tubes}\label{section-isop}
In this section we use similar definitions and notations as in
\cite{Gr04} and \cite{MMP06}. Let $P^m\subset M^n$ be an
$m$-dimensional embedded submanifold of an $n$-dimensional
Riemannian manifold $M$. For a given $p\in P$, we now introduce
Fermi coordinates in a neighborhood $\mathcal {U}'$ of $p$ in $M$.
First we choose normal geodesic coordinates $(y_1,\cdots,y_m)$
centered at $p$ in a neighborhood $\mathcal {U}$ of $p$ in $P$. Then
in $\mathcal {U}$ we fix an orthonormal sections
$E_{m+1},\cdots,E_n$ of the normal bundle $\mathcal {V}$ of $P$ in
$M$ such that they are parallel with respect to the normal
connection along any geodesic ray from $p$ in $P$. The Fermi
coordinates $(x_1,\cdots,x_n)$ of $(\mathcal {U}\subset P\subset)$
$\mathcal {U}'\subset M$ centered at $p$ are defined by
\[x_a\left(exp_{p'}\left(\sum_{j=m+1}^nt_jE_j(p')\right)\right)=y_a(p') \quad \quad (a=1,\cdots,m),\]
\[x_i\left(exp_{p'}\left(\sum_{j=m+1}^nt_jE_j(p')\right)\right)=t_i \quad \quad (i=m+1,\cdots,n),\]
for $p'\in\mathcal {U}$ and any sufficiently small numbers
$t_{m+1},\cdots,t_n$. From the definitions, it is easily seen that
the coordinates vector fields $\partial x_1,\cdots,\partial x_n$
satisfy
\begin{eqnarray}
\nabla_{\partial x_a}\partial x_b|_p\in\mathcal {V}_pP, \quad \quad
\nabla_{\partial x_a}\partial x_i|_p\in\mathcal {T}_pP,\quad \quad\nabla_{\partial x_i}\partial x_j|_{\mathcal {U}}=0,\label{covar-deriv}\\
\langle\partial x_{\alpha}, \partial
x_{\beta}\rangle|_p=\delta_{\alpha\beta}, \quad \quad
\langle\partial x_{a}, \partial x_{i}\rangle|_\mathcal {U}=0,\quad
\quad \langle\partial x_{i}, \partial x_{j}\rangle|_\mathcal
{U}=\delta_{ij}, \label{coeffi ortho}
\end{eqnarray}
where $\nabla$ denotes the covariant derivative in $M$,
$\langle \cdot,\cdot\rangle$ denotes the metric, and the indices convention is
that indices $a,b,c,\cdots\in\{1,\cdots,m\}$, indices
$i,j,k,\cdots\in\{m+1,\cdots,n\}$ and indices
$\alpha,\beta,\gamma,\cdots\in\{1,\cdots,n\}$.

In terms of Fermi coordinates, the distance function, say $\sigma$,
to $P$ in $M$ and the (outward) unit normal vector field, say $N$,
of any tubular hypersurface $P_t$ at a distance $\sigma=t>0$ from
$P$ can be written as
$$
\sigma(q')=\sqrt{\sum_{i=m+1}^nx_i^2} \quad and \quad
N(q')=\sum_{i=m+1}^n\frac{x_i}{\sigma}\frac{\partial}{\partial
x_i}\triangleq\nabla\sigma ,
$$
 for $q'=exp_{p'}\left(\sum_{j=m+1}^nx_jE_j(p')\right)\in P_t\subset\mathcal
{U}'-P$, where the last equality $(\triangleq)$ is just the
\emph{Generalized Gauss Lemma}.

The shape operator $S(t)$ of $P_t$ with respect to $N$ at $q'\in
P_t$ is just the restriction to $P_t$ of the tensorial operator
$S:\mathcal {T}(\mathcal {U}'-P)\rightarrow \mathcal {T}(\mathcal
{U}'-P)$ defined by
\begin{equation}\label{shape def}
SU=-\nabla_UN,
\end{equation}
 for $U\in \mathcal {T}_{q'}(\mathcal {U}'-P)$. It is
easily seen that $S$ is symmetric and $SN=0$. Then covariant
derivative of $S(t)$ along normal geodesic of $P_t$ gives the
\emph{Riccati equation}:
\[S'(t)=S(t)^2+R(t),\]
where $R(t)U=R_{NU}N=(\nabla_{[N,U]}-[\nabla_N,\nabla_U])N$ for
$U\in \mathcal {T}_{q'}P_t$. Taking trace shows that
\begin{equation}\label{Riccati eq}
H'(t)=\parallel S(t)\parallel^2+\rho^M(N,N),
\end{equation} where
$H(t)=Trace(S(t))$ is the mean curvature of $P_t$, $\rho^M$ is the
Ricci curvature of $M$ and $\parallel\cdot\parallel$ is the norm
induced from the metric. Therefore, if the Ricci curvature $\rho^M$
is positive, $H(t)$ is a strictly increasing function as $t$ grows.
Recall that in \cite{GT09}, we showed that at least one is minimal
among a family of properly isoparametric hypersurfaces in a closed
Riemannian manifold. Combining these with the structural result of
\cite{Wa87}, we get
\begin{cor}
There exists a unique minimal isoparametric hypersurface among a
family of properly isoparametric hypersurfaces in a closed
Riemannian manifold of positive Ricci curvature.
\end{cor}
\begin{rem}
In Proposition 4.1 of \cite{GT09}, we found isoparametric functions
which are proper on each Milnor sphere, and Remark 4.2 of
\cite{GT09} told that in that situation the Milnor spheres could
carry metrics of positive Ricci curvature (even non-negative
sectional curvature). So the above corollary tells that there is one
and only one minimal isoparametric hypersurface in each family.
\end{rem}

Now we come to derive the power series expansion formula for $S(t)$
with respect to $t$. For any normal vector $v\in \mathcal {V}_pP$,
the shape operator $T_v$ of $P$ in direction $v$ at $p$ is defined
by
\[\langle T_v(X),Y\rangle=\langle\nabla_XY,v\rangle=-\langle\nabla_XV,Y\rangle|_p~,\] for any
vector fields $X,Y$ tangent to $P$, where $V$ is any normal vector
field with $V|_p=v$. Let $\{T^v_{ab}\}$ denote the coefficients of
$T_v$ under the coordinate vector fields, \emph{i.e.},
$$T_v(\partial x_a|_p)=\sum_{b=1}^mT^v_{ab}\partial x_b|_p, \quad or \quad T^v_{ab}=\langle T_v(\partial x_a),\partial x_b\rangle|_p.$$
Then we will also write $T_v=(T^v_{ab})$ (as a matrix), and
$T_i=T_{\partial x_i|_p}=(T^i_{ab})$ for simplicity. For any unit
normal vector $v=\sum_{j=m+1}^nv_j \partial x_j|_p\in \mathcal
{V}_pP$, \emph{s.t.}, $\sum_{j=m+1}^nv_j^2=1$, we denote by
$\eta_v(t)=exp_p(tv)$ the unique geodesic in direction $v$ through
$p$ in $M$. Obviously, the Fermi coordinates of $\eta_v(t)\in P_t$
are $(0,\cdots,0,tv_{m+1},\cdots,tv_n)$, and
$$N|_{\eta_v(t)}=\eta_v'(t)=\sum_{j=m+1}^nv_j\partial
x_j|_{\eta_v(t)}.$$

Let $g_{\alpha\beta}=\langle\partial x_{\alpha},\partial
x_{\beta}\rangle$ be the metric coefficients. Then we recall their
power series expansion in preparation for that of the shape
operator.
\begin{lem}\label{metric expa}
At the point $\eta_v(t)=exp_p(tv)\in P_t$, the following expansions
hold
\begin{itemize}
\item[]$g_{ab}(t)=\delta_{ab}-2T^v_{ab}t+\left(-\langle R_{v\partial
x_a}v,\partial
x_b\rangle+\sum\limits_cT^v_{ac}T^v_{cb}\right)t^2+O(t^3),$
\item[]$g_{ai}(t)=-\frac{2}{3}\langle R_{v\partial x_a}v,\partial
x_i\rangle
t^2+O(t^3),$
\item[]$g_{ij}(t)=\delta_{ij}-\frac{1}{3}\langle R_{v\partial
x_i}v,\partial x_j\rangle t^2+O(t^3).$
\end{itemize}
\end{lem}
\begin{proof}
Though these formulas can be derived by standard calculations and
some of them have occurred in the literature (cf.\cite{MMP06}), we
briefly prove the second here for the reader's convenience.

By formulas (\ref{covar-deriv}-\ref{coeffi ortho}), $g_{ai}(0)=0$,
$$g_{ai}'(0)=v(\langle\partial x_a,\partial
x_i\rangle)=\langle\nabla_v\partial x_a,\partial
x_i\rangle+\langle\partial x_a,\nabla_v\partial x_i\rangle=0,$$ and
\begin{eqnarray}g_{ai}''(0)&=&\frac{d}{dt}(\langle\nabla_N\partial x_a,\partial
x_i\rangle|_{\eta_v(t)}+\langle\partial x_a,\nabla_N\partial
x_i\rangle|_{\eta_v(t)})|_{t=0}\nonumber\\
&=&v(\langle\nabla_{\widetilde{N}}\partial x_a,\partial
x_i\rangle+\langle\partial x_a,\nabla_{\widetilde{N}}\partial
x_i\rangle) \quad (\emph{where}~~
\widetilde{N}=\sum_{j=m+1}^nv_j\partial
x_j)\nonumber\\
&=&\langle\nabla_v\nabla_{\widetilde{N}}\partial x_a,\partial
x_i\rangle+\langle\partial x_a,\nabla_v\nabla_{\widetilde{N}}\partial x_i\rangle\nonumber\\
&=&-\langle R_{v\partial x_a}v,\partial
x_i\rangle-\frac{1}{3}\langle
R_{v\partial x_i}v,\partial x_a\rangle \quad (\emph{by Lemma 9.20 in \cite{Gr04}})\nonumber \\
&=&-\frac{4}{3}\langle R_{v\partial x_a}v,\partial
x_i\rangle,\nonumber
\end{eqnarray}
where Lemma 9.20 in \cite{Gr04} shows (by properties of Fermi
coordinates and polarization)
\[\nabla_{\partial x_i}\nabla_{\partial x_j}\partial
x_a|_p=-R_{\partial x_i\partial x_a}\partial x_j|_p~,\]
\[\nabla_{\partial x_i}\nabla_{\partial
x_j}\partial x_k|_p=-\frac{1}{3}(R_{\partial x_i\partial
x_j}\partial x_k+R_{\partial x_i\partial x_k}\partial x_j)|_p~.\]

Therefore, by Taylor formula, we get the second expansion formula.
The other two can also be derived similarly.
\end{proof}
\begin{lem}\label{covar of coord fields}
At the point $\eta_v(t)=exp_p(tv)\in P_t$, the following expansions
hold
\begin{eqnarray}
\nabla_{\partial x_a}\partial x_j&=&\sum\limits_b-T^j_{ab}\partial
x_b-t\sum\limits_b\left(\langle R_{v\partial x_a}\partial
x_j,\partial x_b\rangle+\sum\limits_cT^j_{ac}T^v_{cb}\right)\partial
x_b\nonumber\\&&-t\sum\limits_k\langle R_{v\partial x_a}\partial
x_j,\partial x_k\rangle\partial
x_k+\sum\limits_{\alpha}O(t^2)_{\alpha}\partial x_{\alpha},\nonumber\\
\nabla_{\partial x_l}\partial
x_j&=&-\frac{t}{3}\sum\limits_b(\langle R_{v\partial x_l}\partial
x_j,\partial x_b\rangle+\langle R_{v\partial x_j}\partial
x_l,\partial x_b\rangle)\partial
x_b\nonumber\\&&-\frac{t}{3}\sum\limits_k(\langle R_{v\partial
x_l}\partial x_j,\partial x_k\rangle+\langle R_{v\partial
x_j}\partial x_l,\partial x_k\rangle)\partial
x_k+\sum\limits_{\alpha}O(t^2)_{\alpha}\partial x_{\alpha}.\nonumber
\end{eqnarray}
\end{lem}
\begin{proof}By definitions,
$$\langle\nabla_{\partial x_a}\partial x_j, \partial
x_b\rangle|_p=-T^j_{ab},\quad \langle\nabla_{\partial x_a}\partial
x_j,
\partial x_k\rangle|_p=0.$$
Similarly as in the proof of Lemma \ref{metric expa}, we have
\begin{eqnarray}
v\langle\nabla_{\partial x_a}\partial x_j, \partial
x_b\rangle&=&\langle\nabla_v\nabla_{\partial x_a}\partial x_j,
\partial x_b\rangle+\langle\nabla_{\partial x_a}\partial x_j, \nabla_v\partial
x_b\rangle\nonumber\\
&=&-\langle R_{v\partial x_a}\partial x_j,\partial
x_b\rangle+\sum\limits_{c}T^j_{ac}T^{v}_{cb},\nonumber
\end{eqnarray}
\begin{eqnarray}
v\langle\nabla_{\partial x_a}\partial x_j, \partial
x_k\rangle&=&\langle\nabla_v\nabla_{\partial x_a}\partial x_j,
\partial x_k\rangle+\langle\nabla_{\partial x_a}\partial x_j, \nabla_v\partial
x_k\rangle\nonumber\\
&=&-\langle R_{v\partial x_a}\partial x_j,\partial
x_k\rangle.\nonumber
\end{eqnarray}
Suppose that $$\nabla_{\partial x_a}\partial
x_j|_{\eta_v(t)}=\sum\limits_{\alpha}\xi_{\alpha}(t)\partial
x_{\alpha}.$$ Then
$$\langle\nabla_{\partial x_a}\partial x_j, \partial
x_{\beta}\rangle|_{\eta_v(t)}=\sum\limits_{\alpha}\xi_{\alpha}(t)g_{\alpha\beta}(t),$$
and thus by Lemma \ref{metric expa}, we get
\begin{itemize}
\item[] $ \xi_b(t)=-T^j_{ab}-t\left(\langle R_{v\partial x_a}\partial
x_j,\partial
x_b\rangle+\sum\limits_cT^j_{ac}T^v_{cb}\right)+O(t^2),$
\item[] $\xi_k(t)=-t\langle R_{v\partial x_a}\partial x_j,\partial
x_k\rangle+O(t^2),$
\end{itemize}
which completes the proof of the first formula of the lemma.

The second formula can also be derived similarly.
\end{proof}

Now we are ready to give the power series expansion of the shape
operator. From Lemma \ref{covar of coord fields}, one can
immediately deduce the following.
\begin{lem}\label{shape exp1}
At the point $\eta_v(t)=exp_p(tv)\in P_t$, the following expansions
hold
\begin{eqnarray}
\nabla_{\partial x_a}N&=&\sum\limits_b-T^v_{ab}\partial
x_b-t\sum\limits_b\left(\langle R_{v\partial x_a}v,\partial
x_b\rangle+\sum\limits_cT^v_{ac}T^v_{cb}\right)\partial
x_b\nonumber\\&&-t\sum\limits_k\langle R_{v\partial x_a}v,\partial
x_k\rangle\partial
x_k+\sum\limits_{\alpha}O(t^2)_{\alpha}\partial x_{\alpha},\nonumber\\
\nabla_{\partial x_l}N&=&\frac{1}{t}\partial
x_l-\frac{v_l}{t}N-\frac{t}{3}\left(\sum\limits_b\langle
R_{v\partial x_l}v,\partial x_b\rangle\partial
x_b+\sum\limits_k\langle R_{v\partial x_l}v,\partial
x_k\rangle\partial
x_k\right)+\sum\limits_{\alpha}O(t^2)_{\alpha}\partial
x_{\alpha}.\nonumber
\end{eqnarray} \hfill $\Box$
\end{lem}
Notice that along the geodesic $\eta_v(t)$ one can always choose a
system of Fermi coordinates $(x_1,\cdots,x_n)$ such that it
satisfies a further property (Lemma 2.5 in \cite{Gr04}):
$$\partial x_n|_{\eta_v(t)}=\eta_v'(t)=N|_{\eta_v(t)}.$$
Then we denote by $(S_{\alpha\beta})$ the coefficients of the
operator $S$ defined by (\ref{shape def}) under such Fermi
coordinates, \emph{i.e.}, $S(\partial
x_{\alpha})=\sum\limits_{\beta=1}^nS_{\alpha\beta}\partial
x_{\beta}$. Note that $S(\partial
x_n)|_{\eta_v(t)}=S(N)|_{\eta_v(t)}=0$ and $\partial
x_n|_{\eta_v(0)}=\eta_v'(0)=v$ (which implies
$v_{m+1}=\cdots=v_{n-1}=0, v_n=1$). Therefore, from Lemma \ref{shape
exp1}, we obtain
\begin{prop}
At the point $\eta_v(t)=exp_p(tv)\in P_t$, the following expansion
holds
\begin{equation}\label{shape exp2}
S=(S_{\alpha\beta})=\left(\begin{array}{ccc}T_v+tA+\mathcal{O}(t^2)&
tB+\mathcal {O}(t^2)& \mathcal{O}(t^2)
\\tC+\mathcal {O}(t^2)&-\frac{1}{t}I+tD+\mathcal {O}(t^2)& \mathcal{O}(t^2)
\\0& 0& 0 \end{array}\right),
\end{equation}
where $A=(\langle R_{v\partial x_a}v,\partial
x_b\rangle+\sum\limits_cT^v_{ac}T^v_{cb})$ is a matrix with indices
$a,b\in\{1,\cdots,m\}$; $D=(\frac{1}{3}\langle R_{v\partial
x_l}v,\partial x_k\rangle)$ is a matrix with indices
$l,k\in\{m+1,\cdots,n-1\}$; $B=(\langle R_{v\partial x_a}v,\partial
x_k\rangle)$; $C=(\frac{1}{3}\langle R_{v\partial x_l}v,\partial
x_b\rangle)$; $\mathcal {O}(t^2)$ denotes matrices with elements of
order not less than $2$.\hfill $\Box$
\end{prop}
Recall that $S(t)$ is the restriction to $P_t$ of the operator $S$
and as a unit normal vector of $P_t$, $\partial x_n|_{\eta_v(t)}=N$
is an eigenvector of $S$ with corresponding eigenvalue $0$. Thus the
expansion (\ref{shape exp2}) implies
\begin{cor}\label{shape exp3}
The principal curvatures of $P_t$ at $\eta_v(t)$ are just the
eigenvalues of the matrix (with notations as above)
$$
\bar{S}(t):=\left(\begin{array}{cc}T_v+tA+\mathcal{O}(t^2)&
tB+\mathcal {O}(t^2)
\\tC+\mathcal {O}(t^2)&-\frac{1}{t}I+tD+\mathcal {O}(t^2)
\end{array}\right).
$$\hfill $\Box$
\end{cor}
\section{Proof of Theorems and curvature restrictions}\label{proof}
In this section, we apply the expansion formula of the shape
operator (Corollary \ref{shape exp3}) to prove the theorems stated
in the introduction. As byproducts, there are some curvature
restrictions necessary for a Riemannian manifold to admit (certain)
isoparametric functions.

{\bf{Proof of Theorem \ref{minimal focal}:}}

By the structural result of \cite{Wa87}, each regular level
hypersurface $M_s=f^{-1}(s)$ of a transnormal function $f$ is a
tubular hypersurface of either of the focal varieties $M_{\pm}$ (if
non-empty). Then we could apply the results of last section here
with $P=M_{-}$ (alternatively $M_{+}$) and its tubular hypersurface
$P_t$ at distance $t$ ($0<t<dist(M_{-},M_{+})$) from $P$ is just
some regular level hypersurface of $f$. Since now $f$ is an
isoparametric function, $P_t$ has constant mean curvature $H(t)$,
which implies, by Corollary \ref{shape exp3},
\begin{equation}\label{trace}
Trace(\bar{S}(t))=Trace(T_v)-\frac{n-m-1}{t}+t(Trace(A)+Trace(D))+O(t^2)=H(t)
\end{equation}
is independent of the choices of point $p\in P$ and unit normal
vector $v\in\mathcal {V}_pP$. Therefore, comparing the coefficients
of $t^{\lambda}$'s, we know that $Trace(T_v)$ is a constant
independent of the choices of unit normal vector $v\in\mathcal
{V}_pP$ and thus must vanish, since by linearity,
$Trace(T_v)=Trace(T_{-v})=-Trace(T_v)$. This completes the proof of Theorem
\ref{minimal focal}. \hfill $\Box$

Formula (\ref{trace}) in the proof above also implies
\begin{equation}\label{curv restr 1}
\Gamma_P:=Trace(A)+Trace(D)=\frac{1}{3}\left(\rho^N(v,v)+2\sum_{a=1}^mK^N(v,\partial
x_a)+3\parallel T_v\parallel^2\right)
 \end{equation}
 is a constant independent
of the choices of point $p\in P$ and unit normal vector
$v\in\mathcal {V}_pP$, where $\rho^N(v,v)$, $K^N(v,\partial x_a)$
are the Ricci curvature in direction $v$ and the sectional curvature
of the plane spanned by $(v,\partial x_a)$ of $N$ respectively.
Moreover, by taking the Trace of (\ref{curv restr 1}) with respect to
$v$ and by Gauss equation,
\begin{equation}\label{curv restr 2}
\frac{1}{3}\sum_{ij=m+1}^nK^N_{ij}+\sum_{ab=1}^mK^N_{ab}+\sum_{ia}K^N_{ia}-R^P=(n-m)\Gamma_P
\end{equation} is a constant function on $P$, where $K^N_{\alpha\beta}:=K^N(\partial
x_{\alpha},\partial x_{\beta})$, and $R^P$ denotes the scalar
curvature of $P$. In particular, if $dimP=n-1$, or equivalently, if the
isoparametric function $f$ is not proper, then (\ref{curv restr 2})
reduces to
$$
R^N-\rho^N(\nu,\nu)-R^P=\Gamma_P,
$$
 where $\nu$ is the unit normal
vector of $P$. If in addition $N$ is an Einstein manifold,
\emph{i.e.}, $\rho^N\equiv Const$, then $P$ has constant scalar
curvature. In conclusion, we get:
\begin{cor}\label{minimal curv rest}
On the unit normal bundle of the focal varieties $M_{\pm}$ of an
isoparametric function $f$ on a complete Riemannian manifold $N$,
the function $\Gamma_P$ defined in (\ref{curv restr 1}) is constant
and the equality (\ref{curv restr 2}) holds. If in addition $N$ is
Einstein and $f$ is isoparametric and not proper, then the
``non-singular" focal variety is a minimal hypersurface with
constant scalar curvature. \hfill $\Box$
\end{cor}

{\bf{Proof of Theorem \ref{austere focal}:}}

As before, we can apply Corollary \ref{shape exp3} in our case with
$P=M_{-}$ (alternatively $M_{+}$). Now by assumption, $P_t$ has constant principal
curvatures, say $\lambda_1(t),\cdots,\lambda_{n-1}(t)$, with respect to the unit normal vector field in the direction of $\nabla f$. All $\lambda_i(t)$'s depend only on $t>0$, regardless of whether $P_t$ is connected or not. Indeed, when $P$ has codimension $1$, as a level hypersurface $P_t$ may still be disconnected. However, by Footnote \ref{constantcommon}, $\lambda_i(t)$'s depend only on $f|_{P_t}$ which is a one-parameter function of $t$. Then by Corollary \ref{shape exp3},
\begin{equation}\label{shape exp4}
\bar{S}(t)=\left(\begin{array}{cc}T_v& 0
\\0&-\frac{1}{t}I\end{array}\right)+
t\left(\begin{array}{cc}A&B
\\C&D\end{array}\right)+\mathcal {O}(t^2)
\end{equation}
has $\lambda_1(t),\cdots,\lambda_{n-1}(t)$ as its eigenvalues. Let
$\mu_1(v),\cdots,\mu_m(v)$ be the eigenvalues of $T_v$. Then without
loss of generality, by (\ref{shape exp4}) we can assume
$$
\lambda_a(t)=\mu_a(v)+O(t),\quad \lambda_k(t)=-\frac{1}{t}+O(t),
\quad for~ a=1,\cdots,m;~ k=m+1,\cdots,n-1.
$$
 Taking limit of $t$ to
$0$, we have
\begin{equation}\label{austere focal equ}
\mu_a(v)=\lim_{t\rightarrow0+}\lambda_a(t)=:\lambda_a
~(const)\end{equation} for any $p\in P$ and unit normal vector
$v\in\mathcal {V}_pP$, which completes the proof. \hfill $\Box$

Note that for any unit normal vector $v\in\mathcal {V}_pP$, $T_v$
has the same constant eigenvalues $\lambda_1,\cdots,\lambda_m$,
while by linearity, $T_{-v}=-T_v$. Thus the constants $\lambda_a$
must occur in pairs of opposite signs, \emph{i.e.},
$\kappa_1,-\kappa_1,\cdots,\kappa_q,-\kappa_q,0,\cdots,0$. Such
submanifold with principal curvatures in any direction occurring in
pairs of opposite signs is called austere submanifold in
\cite{HL82}. On the other hand, by the Riccati equation
(\ref{Riccati eq}), we know that the Ricci curvature
$\rho^N(\nu,\nu)$ is constant on each $P_t$ (where $\nu$ is the unit
normal vector field) and thus
$$\rho^N(v,v)=\lim_{t\rightarrow0+}\rho^N(\nu,\nu)$$ is constant on
$P$ for any unit normal vector field $v$, which together with
formulas (\ref{curv restr 1}, \ref{austere focal equ}) imply:
\begin{equation}\label{curv restr 4}
\sum_{a=1}^mK^N(v,\partial x_a)\equiv const.
\end{equation}
Thus if $N$ has constant scalar curvature on $P$, so does $P$ by
formulas (\ref{curv restr 2}, \ref{curv restr 4}) and the constancy of
$\rho^N(v,v)$. In conclusion, we obtain:
\begin{cor}\label{austere curv rest}
Suppose that each isoparametric hypersurface $M_t$ has constant
principal curvatures. Then the Ricci curvature $\rho^N$ of $N$ is
constant on each $M_t$ and the focal varieties $M_{\pm}$ in their
normal directions. Moreover, the curvature identity (\ref{curv restr
4}) holds on the focal varieties. If in addition the scalar
curvature of $N$ is constant on $M_{\pm}$, then the focal varieties
also have constant scalar curvature. \hfill $\Box$
\end{cor}
The curvature restrictions above could also be derived by analyzing
the expansion formula (\ref{shape exp4}) which may imply some
subtler curvature relations (perhaps explicit higher order terms
would be required).
\begin{ack}
It is our great pleasure to thank Professor Gudlaugur Thorbergsson
for many useful discussions and kindly introductions of works on
normal holonomy theory, and also thank Professor Yoshihiro Ohnita
for his interests on our work. We would also like to thank the referees for their careful reviews and helpful comments.
\end{ack}


\begin{thebibliography}{123}
\bibitem[BCO03]{BCO03}
J. Berndt, S. Console, and C. Olmos, \emph{Submanifolds and
holonomy},  CRC/Chapman and Hall, Research Notes Series in
Mathematics \textbf{434}. Boca Raton, 2003.

\bibitem[Br99]{Br99}
M. Br\"{u}ck, \emph{Equifocal families in symmetric spaces of
compact type}, J. Reine Angew. Math. \textbf{515} (1999), 73--95.

\bibitem[Car38]{Ca38}
E.~Cartan, \emph{Familles de surfaces isoparam\'etriques dans les
espaces \`a
  courbure constante}, Annali di Mat. \textbf{17} (1938), 177--191.

\bibitem[Car39]{Ca39}
E.~Cartan, \emph{Sur des familles remarquables d'hypersurfaces
  isoparam\'etriques dans les espaces sph\'eriques}, Math. Z. \textbf{45}
  (1939), 335--367.

\bibitem[CCJ07]{CCJ07}
T. E. Cecil, Q. S. Chi, and G. R. Jensen, \emph{Isoparametric
hypersurfaces with four principal curvatures}, Ann. Math.
\textbf{166}
  (2007), no.~1, 1--76.

\bibitem[CDO08]{CDO08}
 S. Console, A. J. Di Scala, and C. Olmos, \emph{A Berger type normal holonomy theorem for complex
 submanifolds}, Math. Ann. \textbf{351} (2011), no.~1, 187--214.

\bibitem[GT09]{GT09}
J.Q. Ge and Z.Z. Tang, \emph{Isoparametric functions and exotic
spheres}, J. Reine Angew. Math., DOI: 10.1515/crelle-2012-0005, March 2012.

\bibitem[GX10]{GX10}
J.Q. Ge and Y.Q. Xie, \emph{Gradient map of isoparametric polynomial and its application to Ginzburg-Landau system}, J. Funct. Anal. \textbf{258} (2010), 1682--1691.

\bibitem[Gr04]{Gr04}
A. Gray, \emph{Tubes}, Second Edition, Progress in Mathematics ,
Vol.221. Birkh$\ddot{a}$user Verlag Basel.Boston.Berlin (2004).

\bibitem[HL82]{HL82}
R. Harvey and B. Lawson, \emph{Calibrated
geometries}, Acta Math. \textbf{148} (1982), 47--157.

\bibitem[Imm08]{Imm08}
S. Immervoll, \emph{On the classification of isoparametric
hypersurfaces with four distinct principal curvatures in spheres},
Ann. Math. \textbf{168}
  (2008), no.~3, 1011--1024.

\bibitem[MMP06]{MMP06}
F. Mahmoudi, R. Mazzeo and F. Pacard, \emph{Constant mean curvature
hypersurfaces condensing on a submanifold}, Geom. Funct. Anal.
\textbf{16} (2006), 924--958.

\bibitem[MP05]{MP05}
R. Mazzeo and F. Pacard, \emph{Foliations by constant mean curvature
tubes}, Comm. Anal. Geom. \textbf{13} (2005), no.4, 633--670.

\bibitem[Mi12]{Mi12}
R. Miyaoka, \emph{Transnormal functions on a Riemannian manifold}, to appear in Diff. Geom. Appl., 2012.

\bibitem[M{\"u}80]{Mu80}
H.F. M{\"u}nzner, \emph{Isoparametric hyperfl\"achen in sph\"aren,
I and II}, Math.
  Ann. \textbf{251} (1980), 57--71 and \textbf{256} (1981), 215--232.

\bibitem[Ni97]{Ni97}
L. Ni, \emph{Notes on Transnormal Functions on Riemannian Manifolds}, unpublished, 1997, available at: http://math.ucsd.edu/~lni/academic/isopara.pdf

\bibitem[No73]{No73}
K. Nomizu, \emph{Some results in E. Cartan's theory of isoparametric
families of hypersurfaces}, Bull. Amer. Math. Soc. \textbf{79},
(1973), 1184-1189.

\bibitem[Th00]{Th00}
G. Thorbergsson, \emph{A survey on isoparametric hypersurfaces and
their generalizations}, In Handbook of differential geometry, Vol.
I, North - Holland, Amsterdam, (2000), 963 - 995.

\bibitem[Wa87]{Wa87}
Q.M. Wang, \emph{Isoparametric Functions on Riemannian Manifolds. I}, Math. Ann. \textbf{277} (1987), 639--646.

\end{thebibliography}
\end{document}